\documentclass[reqno,11pt,twoside]{amsart}

\usepackage{makroEnglish}
\usepackage{url}                

\usepackage{gastex}

\newcommand{\g}{\mathfrak{g}_2}
\newcommand{\got}[1]{\mathfrak{#1}}

\newtheorem{theoremalph}{Theorem}

\begin{document}

\begin{center}
\textbf{\Large{The space of generalized $G_2$-theta functions}}\\
\vspace{1cm}
Chloé GREGOIRE\\
Institut Fourier, University Grenoble I, Grenoble, France\\
E-mail: \url{chloe.gregoire@ujf-grenoble.fr}
\end{center}

\begin{Abstract}
Let $G_2$ be the exceptional Lie group of automorphisms of the complex Cayley algebra and
 $C$ be a smooth, connected, projective curve of genus at least $2$. 
Using the map obtained from extension of structure groups, 
 we prove explicit links between the space of generalized $G_2$-theta functions over $C$ and spaces of generalized theta functions associated to the classical Lie groups $SL_2$ and $SL_3$.
\end{Abstract}

\begin{Keywords}
Generalized theta functions, moduli stacks, principal bundles.
\end{Keywords}
\begin{MathClass}
14H60
\end{MathClass}

\section{Introduction}
Throughout this paper we fix a smooth, connected, projective curve $C$ of genus at least $2$.
For a complex Lie group $G$ we denote by
$\mathcal{M}_C(G)$ the moduli stack of principal $G$-bundles and by $\mathcal{L}$ the ample line bundle that generates the Picard group $\mathrm{Pic}(\mathcal{M}_C(G))$.
The spaces $H^0(\mathcal{M}_C(G),\mathcal{L}^{l})$ of generalized $G$-theta functions of level $l$ are well-known for classical Lie groups but
less understood for exceptional Lie groups. 
Let $G_2$ be the smallest exceptional Lie group: the group of automorphisms of the complex Cayley algebra.
Our aim is to relate the space of generalized $G_2$-theta functions $H^0(\mathcal{M}_C(G_2),\mathcal{L})$ of level one
to other spaces of generalized theta functions associated to classical Lie groups.

Using the Verlinde formula, which gives the dimension of the space of generalized $G$-theta functions for any simple and simply-connected Lie group $G$,
we observe a numerical coincidence: $$\dim H^0(\mathcal{M}_C(SL_2),\mathcal{L}^3)=2^g \dim H^0(\mathcal{M}_C(G_2),\mathcal{L}).$$ In addition, we link together $\dim H^0(\mathcal{M}_C(G_2),\mathcal{L})$ and $\dim H^0(\mathcal{M}_C(SL_3),\mathcal{L})$.
Our aim is to give a geometric interpretation of these dimension equalities.

According to the Borel-De Siebenthal classification \cite{BorelDeSiebenthal},
 the groups $SL_3$ and $SO_4$ appear as the two maximal subgroups of $G_2$ among the connected subgroups of $G_2$ of maximal rank.
 We define two linear maps by pull-back of the corresponding extension maps: on the one hand
 $$H^0(\mathcal{M}_C(G_2),\mathcal{L})\rightarrow H^0(\mathcal{M}_C(SL_3),\mathcal{L})$$ 
 and on the other hand using the isogeny $SL_2\times SL_2\rightarrow SO_4$:
 $$H^0(\mathcal{M}_C(G_2),\mathcal{L})\rightarrow 
H^0(\mathcal{M}_C(SL_2), \mathcal{L})\otimes  H^0(\mathcal{M}_C(SL_2), \mathcal{L}^{3}).$$
These maps take values in the invariant part by the duality involution for the first map and by the action of 
$2$-torsion elements of the Jacobian for the second one.
We denote these invariant spaces by
$H^0(\mathcal{M}_C(SL_3),\mathcal{L})_+$ and 
$\left[H^0(\mathcal{M}_C(SL_2), \mathcal{L}_{SL_2})\otimes  H^0(\mathcal{M}_C(SL_2), \mathcal{L}_{SL_2}^{3})\right]_0            $ respectively.

Using the natural isomorphism proved in \cite{BNR}:
$$H^0(\mathcal{M}_C(SL_3),\mathcal{L})^*\simeq H^0(\mathrm{Pic}^{g-1}(C), 3 \Theta)$$
where $\Theta =\{L \in \mathrm{Pic}^{g-1}(C)\ | \ h^0(C,L)>0 \}$, we prove the following theorem:
\begin{theoremalph}
The linear map 
$$ \Phi :H^0(\mathcal{M}_C(G_2),\mathcal{L})\rightarrow H^0(\mathcal{M}_C(SL_3),\mathcal{L})_+$$
 obtained by pull-back of the extension map $\mathcal{M}_C(SL_3)\rightarrow \mathcal{M}_C(G_2)$ is surjective for a general curve and is an isomorphism when the genus of the curve equals $2$.
\end{theoremalph}

A curve is said satisfying the cubic normality when the multiplication map $\mathrm{Sym}^3 H^0(\mathcal{M}_C(SL_2), \mathcal{L}) \rightarrow H^0(\mathcal{M}_C(SL_2), \mathcal{L}^{3})$ is surjective. 
Using an explicit basis of $H^0(\mathcal{M}_C(SL_2), \mathcal{L}^2)$ described in \cite{BeauvilleII}, we prove the theorem:
\begin{theoremalph}
The linear map 
$$\Psi : H^0(\mathcal{M}_C(G_2),\mathcal{L}_{G_2})\rightarrow 
\left[H^0(\mathcal{M}_C(SL_2), \mathcal{L}_{SL_2})\otimes  H^0(\mathcal{M}_C(SL_2), \mathcal{L}_{SL_2}^{3})\right]_0            $$ 
obtained by pull-back of the extension map $\mathcal{M}_C(SO_4)\rightarrow \mathcal{M}_C(G_2)$
is an isomorphism for a general curve satisfying the cubic normality. 
\end{theoremalph}

\begin{nota}
We use the following notations:
\begin{itemize}
\item
$C$ a smooth, connected, projective curve of genus $g$ at least 2,
\item
$G$ a connected and simply-connected simple complex Lie group,
\item
$\mathcal{M}_C(G)$ the moduli stack of principal $G$-bundles on $C$,
\item
$K_C$ the canonical bundle on $C$ ,
\item
 $\mathrm{Pic}^{g-1}(C)$ the Picard group parametrizing line bundles on $C$ of degree $g-1$,
 \item
 $h^0(X,\mathcal{L})=\dim H^0(X,\mathcal{L})$.
\end{itemize}
\end{nota}

\section{Principal $G_2$-bundles arising from vector bundles of rank two and three}

\subsection{The octonions algebra}
Let $\OO$ be the complex algebra of the octonions, $\mathrm{Im}(\OO)$ the $7$-dimensional subalgebra of the imaginary part of $\OO$ and $\mathcal{B}_0=(e_1, \dots, e_7)$ the canonical basis of $\mathrm{Im}(\OO)$(see \cite{Baez}).
The exceptional Lie group $G_2$ is the group of automorphisms of the octonions.

In Appendix \ref{appendix The octonion algebra} we give the multiplication table in $\mathrm{Im}(\OO)$ by the Fano diagram and we introduce another basis $\mathcal{B}_1=(y_1, \dots, y_7)$ of $\mathrm{Im}(\OO)$ 
so that $\langle y_1, y_2, y_3\rangle$ and $\langle y_4, y_5, y_6 \rangle$ are isotropic and orthogonal and $y_7$ is orthogonal to both of these subspaces. This basis is defined in Appendix \ref{appendix The octonion algebra}, as well as two other basis obtained by permutation of elements of $\mathcal{B}_1$:
$\mathcal{B}_2=(y_2,y_3,y_4,y_5,y_6,y_1,y_7)$ and $\mathcal{B}_3=(y_1,y_2,y_4,y_5,y_3,y_6,y_7)$.

In the following paragraphs, we use local sections of rank-$7$ vector bundles satisfying multiplication rules of $\mathcal{B}_2$ or $\mathcal{B}_3$.

\subsection{Principal $G_2$-bundles admitting a reduction}
We introduce the notion of non-degenerated trilinear form on $\mathrm{Im}(\OO)$  as Engel did it in \cite{Engel1} and\cite{Engel2}.
A trilinear form $\omega$ on $\mathrm{Im}(\OO)$ is said non-degenerated if the associated bilinear symmetric form $B_\omega$ is non-degenerated where $B_\omega(x,y) =\omega(x,\cdot,\cdot)\wedge \omega(y,\cdot,\cdot)\wedge \omega(\cdot,\cdot,\cdot)$ $\forall x, y \in \mathrm{Im}(\OO)$.
\begin{lem}
Giving a principal $G_2$-bundle is equivalent to giving a rank-$7$ vector bundle with a non-degenerated trilinear form.
\end{lem}

\begin{proof}
Let $P$ be a principal $G_2$-bundle on $C$ and $V$ the associated rank-$7$ vector bundle and let $\omega$ be any non-degenerated trilinear form on $\mathrm{Im}(\OO)$.
By construction, $V$ has a reduction to $G_2$, \textit{i.e.} it exists a section $\sigma : C\rightarrow GL_7/G_2$.
The Lie group $G_2$ is the stabilizer $\mathrm{Stab}_{SL_7}(\omega)$ under the action of $SL_7$. Besides, under the action of $GL_7$, $\mathrm{Stab}_{GL_7}(\omega_0)\simeq G_2 \times \mathbb{Z}/3\mathbb{Z}$. Then
$$ \sigma : C\stackrel{\sigma}{\rightarrow} GL_7/G_2 \twoheadrightarrow GL_7/(G_2 \times \mathbb{Z}/3\mathbb{Z})\simeq  GL_7/ \mathrm{Stab}_{GL_7}(\omega) \simeq \mathrm{Orb}_{GL_7}(\omega).$$
In addition, the orbit $\mathrm{Orb}_{GL_7}(\omega)$ is the set of all the non-degenerated trilinear form on $\mathrm{Im}(\OO)$.
So, $V$ is fitted with a non-degenerated trilinear form.
Reciprocally any rank-$7$ vector bundle fitted with a non-degenerated trilinear form defines a $G_2$-vector bundle.
\end{proof}

For a principal $G_2$-bundle, we use the non-degenerate trilinear form $\omega$ on $\mathcal{V}$, locally defined by
$\omega(x,y,z)=-\mathrm{Re}[(xy)z]$.

According to the Borel-De Siebenthal classification (see \cite{BorelDeSiebenthal}), $SL_3$ and $SO_4$ are, up to conjugation, the two maximal subgroups of $G_2$ among the connected subgroups of $G_2$ of maximal rank.
Using the inclusion $G_2 \subset SO_7=SO(\mathrm{Im}(\OO))$
both of the following lemma describe the rank-$7$ vector bundle (and the non-degenerate trilinear form) associated to a principal $G_2$-bundle which admits either a $SL_3$-reduction or a $SO_4$-reduction.

 Note that $\mathrm{M}_C(SO_4)$ has two connected components distinguished by the second Stiefel Whitney class.
We only make here explicit computations with regards to the connected component $\mathrm{M}_C^+(SO_4)$ of $\mathrm{M}_C(SO_4)$ containing the trivial bundle.

\subsubsection{Principal $G_2$-bundles arising from rank-$3$ vector bundles}

\begin{lem}
\label{Lem Inclusion SL3 dans G2 and associated vector bundle}
Let $E$ be a rank-$3$ vector bundle with trivial determinant and let 
$E(G_2)$ be his associated principal $G_2$-bundle and $\mathcal{V}$ be his associated rank-$7$ vector bundle.
Then, $\mathcal{V}$ has the following decomposition and the local sections basis $\mathcal{B}_2$ is adapted to this decomposition:
$$\mathcal{V}=E\oplus E^* \oplus \mathcal{O}_C.$$
 The non-degenerate trilinear form $\omega$ is defined by the following local conditions:
\begin{enumerate}
\item 
$\Lambda^3 E\simeq \Lambda^3 E^* \simeq \C$ and 
$\omega (y_2,y_3,y_4)=\omega (y_5,y_6,y_1) =-\sqrt{2}$.
\item
On $E\times E^*\times\mathcal{O}_C$:
$$\omega (y_2,y_5,y_7)=\omega (y_3,y_6,y_7)=\omega (y_4,y_1,y_7) =i,$$
\item
All other computation, not obtainable by permutation of the previous triplets, equals zero.
\end{enumerate}

\end{lem} 
\begin{proof}
Under the action of $SL_3$, $\mathrm{Im}(\OO)$ decomposes into $SL_3$-modules:
$$
\begin{array}{rcl}
\mathrm{Im}(\OO)& =& \langle y_2 , y_3 , y_4 \rangle \oplus \langle y_5 , y_6 , y_1\rangle \oplus  \langle y_7\rangle\\
 &=&\C^3 \oplus (\C^3)^* \oplus \C 
\end{array}
$$
where $\{y_2 , y_3 , y_4, y_5 , y_6 , y_1, y_7\}$ is the basis $\mathcal{B}_2$ defined in Appendix \ref{appendix The octonion algebra};
$\langle y_2 , y_3 , y_4\rangle$ is an isotropic subspace of dimension $3$, dual of $\langle y_5 , y_6 , y_1\rangle$. 

So, the $7$-rank vector bundle $\mathcal{V}$ associated to a rank-$3$ vector bundle $E$ is:
$$
\begin{array}{rl}
\mathcal{V} = & E \stackrel{SL_3}{\times}(\C^3 \oplus (\C^3)^* \oplus \C),\\
\mathcal{V}= &  E\oplus E^* \oplus \mathcal{O}_C.
\end{array}
$$
Evaluations given for $\omega$ are deduced from Table \ref{tabular B2} of Appendix \ref{appendix The octonion algebra}. 
%
%
%
\end{proof}

\subsubsection{Principal $G_2$-bundles arising from two rank-$2$ vector bundles} 
 The following lemma makes explicit the vector bundle associated to a principal $G_2$-bundle extension of an element of $\mathrm{M}_C^+(SO_4)$, using the surjective map $\mathrm{M}_C(SL_2)\times \mathrm{M}_C(SL_2)\twoheadrightarrow \mathrm{M}_C^+(SO_4)$.


\begin{lem}
\label{Lem Inclusion SL2*SL2 dans G2 and associated vector bundle}
Let $E, F$ be two rank-$2$ vector bundles of trivial determinant.
Denote by $(E,F)$ the associated principal $SO_4$-bundle, $P$ the associated principal $G_2$-bundle and $\mathcal{V}$ the associated rank-$7$ vector bundle.
Then, $\mathcal{V}$ has the following decomposition and the local sections basis $\mathcal{B}_3$ is adapted to this decomposition:
$$\mathcal{V}=E^*\otimes F \oplus \mathrm{End}_0(F).$$
The non-degenerate trilinear form $\omega$ on $\mathcal{V}$
 is defined by the following local conditions:
\begin{enumerate}
\item 
On $(E^*\otimes F)^3$ and on $(E^*\otimes F )\times(\mathrm{End}_0(F))^2$, $\omega$ is identically zero,
\item
On $(E^*\otimes F)^2 \times\mathrm{End}_0(F)$:
$$
\begin{array}{rcl}
\omega (y_2,y_4,y_3)&=\omega (y_5,y_1,y_6) &=\sqrt{2},\\
\omega (y_4,y_1,y_7) &=\omega (y_2,y_5,y_7)&=i.
\end{array}
$$
\item
$\Lambda^3\mathrm{End}_0(F) \simeq \C$ and $\omega (y_3,y_6,y_7) =i,$
\item
All other computation, not obtainable by permutation of the previous triplets, equals zero.
\end{enumerate}

\end{lem}
\begin{proof}
The groups $SO_4$ and $SL_2\times SL_2$ are isogenous.
 Under the action of $SO_4$, $\mathrm{Im}(\OO)$ has the following decomposition:
$$
\begin{array}{rcl}
\mathrm{Im}(\OO)&=& \langle y_1,y_2,y_4,y_5\rangle \oplus \langle y_3,y_6,y_7\rangle, \\
&\simeq& M^*\otimes N \stackrel{\bot}{\oplus} \mathrm{End}_0(N)
\end{array}
$$
where $M,N$ are $2$-dimensional; an element $\overline{(A,B)}$ of $SO_4$ ($A,B \in SL_2$) acts on $ M^*\otimes N$ by $A\otimes B$ and by conjugation $\mathrm{End}_0(N)$.

So, the rank-$7$ vector bundle $\mathcal{V}$ associated to $(E,F)(G_2)$, when $E,F$ are two rank-$7$ vector bundle of trivial determinant, is:
$$\mathcal{V}=E^*\otimes F \oplus \mathrm{End}_0(F).$$
Evaluations given for $\omega$ are deduced from Table \ref{tabular B3} of Appendix \ref{appendix The octonion algebra}.
\end{proof}

\section{Equalities between dimensions of spaces of generalized theta functions}

Here are some dimension counts, using the Verlinde Formula, to calculate 
$h^0(\mathcal{M}_C(G_2),\mathcal{L})$ and $h^0(\mathcal{M}_C(SL_2),\mathcal{L}^3)$.
%
\begin{prop}
\label{Prop h0(G2,L) h0(SL2,L3)}
Dimensional equalities between the following spaces of generalized theta functions
occur:
\begin{align}
h^0(\mathcal{M}_C(G_2),\mathcal{L}) &= \left(\frac{5+\sqrt{5}}{2}\right)^{g-1} +\left(\frac{5-\sqrt{5}}{2}\right)^{g-1},
\label{h0(G2,L)}\\
h^0(\mathcal{M}_C(SL_2),\mathcal{L}^3)&=  2^g\left[\left(\frac{5+\sqrt{5}}{2}\right)^{g-1}+
\left(\frac{5-\sqrt{5}}{2}\right)^{g-1}\right] ,\label{h0(SL2,L3)}\\
\text{so, }\quad h^0(\mathcal{M}_C(SL_2),\mathcal{L}^3)&=2^g h^0(\mathcal{M}_C(G_2),\mathcal{L}). \label{h0(SL2,L3)=2g h0(G2,L)}
\end{align}
\end{prop}

\begin{proof}
See Appendix \ref{appendix h0(G2,L)} and \ref{appendix h0(SL2,L3)}
\end{proof}

\section{Surjectivities and isomorphisms between $H^0(\mathcal{M}_C(G_2),\mathcal{L})$ and $H^0(\mathcal{M}_C(SL_3),\mathcal{L})_+$}

To avoid confusion, we sometimes specify the group $G$ in the notation of the generator of the Picard group $\mathcal{L}$ writing $\mathcal{L}_G$. 

Let $ H^0(\mathcal{M}_C(SL_3),\mathcal{L})_+$ be the invariant part of $H^0(\mathcal{M}_C(SL_3),\mathcal{L})$ by the duality involution:
 the eigenspace of $H^0(\mathcal{M}_C(SL_3),\mathcal{L})$
associated to the eigenvalue $1$ under the natural involution $E \mapsto \sigma (E)=E^*$.

We consider the extension map
$\mathrm{i} : \mathcal{M}_C(SL_3)\rightarrow \mathcal{M}_C(G_2)$
which associates to a rank-$3$ vector bundle of trivial determinant the associated principal $G_2$-bundle.
The pull-back $\mathrm{i}^*(\mathcal{L}_{G_2})$ equals $\mathcal{L}_{SL_3}$.
\begin{thm}
The extension map
$\mathrm{i} : \mathcal{M}_C(SL_3)\rightarrow \mathcal{M}_C(G_2)$
induces by 
pull-back a linear map between the following spaces of generalized theta functions:
$$H^0(\mathcal{M}_C(G_2),\mathcal{L})\rightarrow H^0(\mathcal{M}_C(SL_3),\mathrm{i}^*(\mathcal{L}_{G_2})).$$
This map takes values in $ H^0(\mathcal{M}_C(SL_3),\mathcal{L})_+$.

We denote by $\Phi$ this map:
$$ \Phi :H^0(\mathcal{M}_C(G_2),\mathcal{L})\rightarrow H^0(\mathcal{M}_C(SL_3),\mathcal{L})_+.$$
\end{thm}

\begin{proof}
\begin{enumerate}
\item
The pull-back $\mathrm{i}^*(\mathcal{L}_{G_2})$ equals $\mathcal{L}_{SL_3}$.Indeed, let $E$ be a rank-$3$ vector bundle.
By Lemma \ref{Lem Inclusion SL3 dans G2 and associated vector bundle}, the rank-$7$ vector bundle associated to $E$ is $E\oplus E^*  \oplus \mathcal{O}_C$.
We study the commutative diagram:
\begin{diagram}
\mathcal{M}_C(G_2)		& \rTo^{\rho_1} 	& \mathcal{M}_C(SL_7)\\
\uTo^{\mathrm{i}}				& 	& \uTo_{\rho_2}\\
\mathcal{M}_C(SL_3) 		& \rTo^{\rho_3}			& \mathcal{M}_C(SL_3)\times \mathcal{M}_C(SL_3)
\end{diagram}
where $\mathrm{i}$ and $\rho_1$ are maps of extension of group of structure and $\forall E \in \mathcal{M}_C(SL_3) ,$
$\rho_3(E)=(E, E^*)$ and  $\forall (E,F) \in  \mathcal{M}_C(SL_3)\times \mathcal{M}_C(SL_3)$, 
$\rho_2 (E,F) =E \oplus F \oplus \mathcal{O}_C$.
By Proposition $2.6$ of \cite{LaszloSorger}, applied with $SL_7$, $G_2$ 
 and the irreducible representation $\rho_1$ of highest weight $\varpi_1$, the Dynkin index of $d(\rho_1)$ equals $2$.
Therefore, $\rho_1^*(\mathcal{D}_{SL_7})=\mathcal{L}_{G_2}^{2}$ where $\mathcal{D}_{SL_7}$ the determinant bundle generator of $\mathrm{Pic}(\mathcal{M}_C(SL_7))$ (see \cite{KNR} and \cite{LaszloSorger}).
In addition, by the same proposition,
$\rho_2^*(\mathcal{D}_{SL_7})=\mathcal{L}_{SL_3} \boxtimes  \mathcal{L}_{SL_3}$
and
$\rho_3^*(\rho_2^*(\mathcal{D}_{SL_7}))=\mathcal{L}_{SL_3}^{ 2}$.
So, $\mathrm{i}^*(\mathcal{L}_{G_2})^{2}=\mathcal{L}_{SL_3}^{ 2}$ so that $\mathrm{i}^*(\mathcal{L}_{G_2})=\mathcal{L}_{SL_3}$ since the Picard group
$\mathrm{Pic}(\mathcal{M}_C(SL_3))$ is isomorphic to $\mathbb{Z}$.

\item
We show that the image of the linear map $\Phi$ is contained in $H^0(\mathcal{M}_C(SL_3),\mathcal{L})_+$.

The morphism $\mathrm{i}$ is $\sigma$-invariant: for all $ E\in \mathcal{M}_C(SL_3)$, the $G_2$-principal bundles $E(G_2)$ and $E^*(G_2)$ are isomorphic.
Indeed, the Weyl group $ W(SL_3)$ is contained in $W(G_2)$ since so are there normalizer; and $ W(SL_3)$ is a subgroup of $W(G_2)$ of index $2$.
We consider $\overline{g}$ in the Weyl group $W(G_2)\backslash W(SL_3)$ and $g \in G_2$ a representative of the equivalence class of $\overline{g}$. Then, $g\notin SL_3$. 
Let $C_g$ be the inner automorphism of $G_2$ induced by $g$. As the subalgebra $\got{sl}_3$ of $\got{g}_2$ corresponds to the long roots and as each element of the Weyl group $W(G_2)$ respects the Killing form on $\got{g}_2$, $C_g(SL_3)$ is contained in $SL_3$. The restriction of $C_g$ to $SL_3$ is then an exterior automorphism of $SL_3$ 
which we call
$\alpha : SL_3 \rightarrow SL_3$.
This automorphism exchanges the two fundamental representations of $SL_3$.
So, $\alpha$ induces an automorphism $\widetilde{\alpha}$ on $\mathcal{M}_C(SL_3)$ such that, $\forall E\in \mathcal{M}_C(SL_3)$, $\widetilde{\alpha}(E)=E^*$. Consider the following commutative diagram, where $\widetilde{C}_{g}$ is
the inner automorphism given by $g$:
\begin{diagram}
\mathcal{M}_C(SL_3) &\rInto &\mathcal{M}_C(G_2) \\
\dTo^{\widetilde{\alpha}}&&\dTo_{\widetilde{C_{g}}}\\
\mathcal{M}_C(SL_3) & \rInto & \mathcal{M}_C(G_2).
\end{diagram}
Then, $\forall E\in \mathcal{M}_C(SL_3)$,
$$E^*(G_2)=\widetilde{\alpha}(E)(G_2)= \widetilde{C_{g}}(E(G_2))\simeq E(G_2)$$
since $ \widetilde{C_{g}}$ is an inner automorphism.
Thus, $\mathrm{i}(E)$ and $\mathrm{i}(\sigma(E))$ are isomorphic.

The $\sigma$-invariance of $\mathrm{i}$ implies that the image of $\Phi$ is contained in one of the two eigenspaces of $H^0(\mathcal{M}_C(SL_3),\mathcal{L})$. 
As $\sigma^*(\mathcal{L}_{SL_3})\simeq\mathcal{L}_{SL_3}$, which is the isomorphism which implies identity over the trivial bundle, we get $\sigma(\mathrm{i}^*(\mathcal{L}_{G_2}))=\sigma^*(\mathcal{L}_{SL_3})=\mathcal{L}_{SL_3}=\mathrm{i}^*(\mathcal{L}_{G_2})$.
Thus, the image of $\Phi$ is contained in $H^0(\mathcal{M}_C(SL_3),\mathcal{L})_+$, the eigenspace relative to the eigenvalue $1$.
\end{enumerate}
\end{proof}

We remind two points of vocabulary:
an \textit{even theta-characteristic} $\kappa$ on a curve $C$ is an element $\kappa$ of $\mathrm{Pic}^{g-1}(C)$ such that $\kappa \otimes \kappa = K_C$ and $h^0(C, \kappa)$ is even ; 
a curve $C$ is said \textit{without effective theta-constant} if $h^0(C, \kappa)=0$ for all even theta-characteristic $\kappa$.
The set of all even theta-characteristics is named
$\Theta^{\mathrm{even}}(C)$.

\begin{thm}
\label{Thm Application Phi}
The linear map $\Phi :H^0(\mathcal{M}_C(G_2),\mathcal{L})\rightarrow H^0(\mathcal{M}_C(SL_3),\mathcal{L})_+$
\begin{enumerate}
\item
\label{thm item Phi surjective for C without eff theta-constant}
is surjective when the curve $C$ is without effective theta-constant.
\item
is an isomorphism if the genus of $C$ equals $2$.
\end{enumerate}
\end{thm}

\begin{proof}
\begin{enumerate}
\item
Let $C$ be a curve without effective theta-constant and consider the following diagram:
\begin{diagram}
\mathcal{M}_C(G_2)		& \rTo^{\rho_1} 	& \mathcal{M}_C(SL_7)\\
\uTo^{\mathrm{i}}		&    \ruTo_{\rho} &    \\
\mathcal{M}_C(SL_3) 	&	            & 
\end{diagram}
We introduce the element $\Delta_\kappa$ defined for each even theta-characteristic~$\kappa$:
$$ \Delta_\kappa =\{P\in \mathcal{M}_C(SL_7)\ | \ h^0(C, P(\C^7)\otimes \kappa) > 0\}. $$
These $ \Delta_\kappa$ are Cartier divisors, so they define, up to a scalar, an element of $ H^0(\mathcal{M}_C(SL_7), \mathcal{L})$.
The image $\Phi(\rho_1^*(\Delta_\kappa))=\rho^*(\Delta_\kappa)$ is
$$
\begin{array}{rl}
\rho^*(\Delta_\kappa)=&\{E\in \mathcal{M}_C(SL_3)\ |\  h^0(C,E\oplus E^* \oplus \mathcal{O}_C) \otimes \kappa) >0 \},\\
=&\{E\in \mathcal{M}_C(SL_3)\ |\ h^0(C,E\otimes \kappa)+h^0(C,E^*\otimes \kappa)+  h^0(C,\kappa) >0 \},\\
=&\{E\in \mathcal{M}_C(SL_3)\ |\ h^0(C,E\otimes \kappa)+h^0(C, E^*\otimes \kappa) >0 \},\\
&\text{because $C$ is without effective theta-constant,}\\
=&\{E\in \mathcal{M}_C(SL_3)\ |\ 2h^0(C,E\otimes \kappa) >0 \}, \text{ by Serre duality}.\\
\end{array}
$$
Thus, $\rho^*(\Delta_\kappa)=2H_\kappa$
where $H_\kappa:=\{ E \in \mathcal{M}_C(SL_3) \ | \   h^0(C,\mathcal{E}\otimes \kappa)>0\}$.
Therefore, to show the surjectivity of $\Phi$, it suffices to show that $\{H_\kappa\ | \ \kappa \in \Theta^{\mathrm{even}}(C)\}$
generates $H^0(\mathcal{M}_C(SL_3),\mathcal{L})_+$.
We consider 
$$\Theta =\{L \in \mathrm{Pic}^{g-1}(C)\ | \ h^0(C,L)>0 \}$$ and the natural map between the spaces $H^0(\mathcal{M}_C(SL_3),\mathcal{L})^*$ and $H^0(\mathrm{Pic}^{g-1}(C), 3 \Theta)$.
By Theorem $3$ of \cite{BNR}, this map is an isomorphism and, besides, it is 
equivariant for the two involutions on $H^0(\mathcal{M}_C(SL_3),\mathcal{L})^*$ and $H^0(\mathrm{Pic}^{g-1}(C), 3 \Theta)$ (respectively $E \mapsto E^*$ and $L\mapsto K_C\otimes L^{-1}$).
So, the components $(+)$ and $(-)$ of each part are in correspondence: $ H^0(\mathcal{M}_C(SL_3),\mathcal{L})^*_+$ is isomorphic to $H^0(\mathrm{Pic}^{g-1}(C), 3 \Theta)_+$.
Denote by $\varphi$ this isomorphism
$\varphi : \PP H^0(\mathcal{M}_C(SL_3),\mathcal{L})_+ \stackrel{\sim}{\rightarrow} \PP H^0(\mathrm{Pic}^{g-1}(C), 3 \Theta)_+^* $. For all even theta-characteristic $\kappa$, the image $\varphi(H_\kappa)$ is $\varphi_{3\Theta}(\kappa)$ where $\varphi_{3\Theta}$ is the following map:
$$\varphi_{3\Theta} : \mathrm{Pic}^{g-1}(C) \rightarrow \PP H^0(\mathrm{Pic}^{g-1}(C), 3 \Theta)_+^*=|3 \Theta|^*_+ .$$
The set $\{ \varphi_{3\Theta}(\kappa)\ | \ \kappa \in \Theta^{\mathrm{even}}(C)\}$
generates $|3 \Theta|^*_+$.
Indeed, in the following commutative diagram
\begin{diagram}
						&					& |4\Theta|^*_+			\\
						&\ruDashto^{\varphi_{4\Theta}}				&   \dDashto   	\\
\mathrm{Pic}^{g-1}(C) 	&	\rDashto^{\varphi_{3\Theta}}			& 	| 3\Theta|^*_+	,
\end{diagram}
the map $|4\Theta|^*_+	\dashrightarrow |
3\Theta|^*_+$ is surjective because it is induced by the inclusion
$D \in H^0(C,3\Theta)_+ \mapsto D+\Theta \in H^0(C,4\Theta)_+$. In addition, by 
\cite{SermanKopeliovichPauly}, when $C$ is without effective theta-constant, 
$\{\varphi_{4\Theta}(\kappa)\ |\ \kappa \in \Theta^{\mathrm{even}}(C)\}$
 is a base of $|4\Theta|^*_+$ (the number of even theta-characteristics equals $2^{g-1}(2^g+1)$ 
 which equals the linear dimension of $|4\Theta|^*_+$).
Thus, 
$\{  \varphi_{3\Theta}(\kappa)\ |\ \kappa\in \Theta^{\mathrm{even}}(C)\}$
generates $|3 \Theta|^*_+$ and $\{H_\kappa\ |\ \kappa\in \Theta^{\mathrm{even}}(C)\}$ generates the space $H^0(\mathcal{M}_C(SL_3),\mathcal{L})_+$.
As we have shown that $H_\kappa$ equals $\Phi(\rho_1(\Delta_\kappa))$ for all $\kappa$ even theta-characteristic, the map $\Phi$ is surjective.

\item
By \cite{BNR}, the dimension of $ H^0(\mathcal{M}_C(SL_3),\mathcal{L})^*_+$ equals the dimension of $H^0(\mathrm{Pic}^{g-1}(C), 3 \Theta)_+$, that is $\frac{3^g+1}{2}$.
When the genus of $C$ is $2$, the dimension of $H^0(\mathrm{Pic}^{g-1}(C), 3 \Theta)_+$ equals $5$ which is also the dimension of $H^0(\mathcal{M}_C(G_2),\mathcal{L})$ by Proposition \ref{Prop h0(G2,L) h0(SL2,L3)}.\eqref{h0(G2,L)}. 
A curve of genus $2$ is without effective theta-constant. So, by this dimension equality and by the point \eqref{thm item Phi surjective for C without eff theta-constant} of the theorem, $\Phi$ is an isomorphism when the genus of $C$ equals $2$.
\end{enumerate}
\end{proof}

\section{Isomorphisms between spaces of generalized $G_2$-theta functions and generalized $SL_2$-theta functions}

Let $\mathrm{JC}[2]$ be the group of $2$-torsion elements of the Jacobian: $\mathrm{JC}[2]=\{\alpha \in \mathrm{Pic}^0(C)\ |\ \alpha \otimes \alpha =\mathcal{O}_C\}$.
This group acts on $\mathcal{M}_{SL_2} \times \mathcal{M}_{SL_2}$ : for $ \alpha\in \mathrm{JC}[2]$ and
$(E,F)\in \mathcal{M}_{SL_2} \times \mathcal{M}_{SL_2} $, we associate $(E\otimes \alpha,F\otimes \alpha)$.
Let $\left[H^0(\mathcal{M}_C(SL_2), \mathcal{L}_{SL_2})\otimes  H^0(\mathcal{M}_C(SL_2), \mathcal{L}_{SL_2}^{3})\right]_0 $ be the invariant part 
of $\left[H^0(\mathcal{M}_C(SL_2), \mathcal{L}_{SL_2})\otimes  H^0(\mathcal{M}_C(SL_2), \mathcal{L}_{SL_2}^{3})\right]$
under the action of the element of the Jacobian group $\mathrm{JC}[2]$.

We study the subgroup $SO_4$ of $G_2$, which is isogenous to $SL_2\times SL_2$, and the linear map induced by pull-back by the extension map $\mathrm{j} : \mathcal{M}_C(SO_4)\rightarrow \mathcal{M}_C(G_2)$
which associates to two rank-$2$ vector bundle of trivial determinant the associated principal $G_2$-bundle.
The pull-back $j^*(\mathcal{L}_{G_2})$ equals $\mathcal{L}_{SL_2}\boxtimes\mathcal{L}_{SL_2}^{3}$
\begin{thm}
\label{Theorem definition Psi}
The extension map
$\mathrm{j} : \mathcal{M}_C(SO_4)\rightarrow \mathcal{M}_C(G_2)$
induces by pull-back a linear map between the following spaces of generalized theta functions: 
$$ H^0(\mathcal{M}_C(G_2),\mathcal{L})\rightarrow H^0(\mathcal{M}_C(SO_4),j^*(\mathcal{L}_{G_2})).$$
This map takes values in $\left[H^0(\mathcal{M}_C(SL_2), \mathcal{L}_{SL_2})\otimes  H^0(\mathcal{M}_C(SL_2), \mathcal{L}_{SL_2}^{3})\right]_0 $.

We denote by $\Psi$ this map:
$$\Psi : H^0(\mathcal{M}_C(G_2),\mathcal{L}_{G_2})\rightarrow 
\left[H^0(\mathcal{M}_C(SL_2), \mathcal{L}_{SL_2})\otimes  H^0(\mathcal{M}_C(SL_2), \mathcal{L}_{SL_2}^{3})\right]_0             $$ 
\end{thm}

\begin{proof}
\begin{enumerate}
\item
Consider the following commutative diagram:
\begin{diagram}
\mathcal{M}_C(G_2)	 		& \rTo^{\rho_1} 					 		& \mathcal{M}_C(SL_7)\\
\uTo^{\mathrm{j}}					&											& \uTo_{\rho_4} \\
\mathcal{M}_C(SL_2)\times \mathcal{M}_C(SL_2)& \rTo^{(f_1,f_2)}			&\mathcal{M}_C(SL_4)\times \mathcal{M}_C(SL_3)\\
\end{diagram}
where $\mathrm{j}$ and $\rho_1$ are the extension maps, $f_1(M,N)=M^*\otimes N $,
$f_2(M,N)= \mathrm{End}_0(N)$ and $\rho_4(A,B)=A\oplus B$.

As in the previous section, we calculate explicitly $j^*(\mathcal{L}_{G_2})$.

Let $\mathcal{D}_{SL_7}$ be the determinant bundle of $\mathcal{M}_C(SL_7)$ and $\mathrm{pr}_1$ and $\mathrm{pr}_2$
the canonical projections of
$SL_2\times SL_2$. We get
$$f_1^*(\mathcal{L}_{SL_4})=\mathrm{pr}_1^*(\mathcal{L}_{SL_2})^{2}\otimes \mathrm{pr}_2^*(\mathcal{L}_{SL_2})^{2}=\mathcal{L}_{SL_2}^{2} \boxtimes \mathcal{L}_{SL_2}^{2}$$
and according to Table $B$ of \cite{SorgerModuliofPrincipalbundles}:
$$f_2^*(\mathcal{L}_{SL_3})=\mathrm{pr}_2^*(\mathcal{L}_{SL_2})^{4},$$
since $f_2$ is associated to the adjoint representation of $SL_2$, which has Dynkin index $4$. 
$$
\begin{array}{rrcl}
&\rho_4^*(\mathcal{D})&=&\mathcal{L}_{SL_4}\boxtimes \mathcal{L}_{SL_3},\\
\text{so}\quad&j^*(\mathcal{L}_{G_2}^{2})&=&(f_1,f_2)^*(\mathcal{L}_{SL_4}\boxtimes \mathcal{L}_{SL_3}),\\
&&=&f_1^*(\mathcal{L}_{SL_4})\otimes f_2^*(\mathcal{L}_{SL_3}),\\
&&=&[\mathrm{pr}_1^*(\mathcal{L}_{SL_2})^{2}\otimes \mathrm{pr}_2^*(\mathcal{L}_{SL_2})^{2}]\otimes 
\mathrm{pr}_2^*(\mathcal{L}_{SL_2})^{4},\\
&j^*(\mathcal{L}_{G_2}^{2})&=&\mathrm{pr}_1^*(\mathcal{L}_{SL_2})^{2}\otimes \mathrm{pr}_2^*(\mathcal{L}_{SL_2})^{6},\\
\text{so}\quad&j^*(\mathcal{L}_{G_2})&=&\mathrm{pr}_1^*(\mathcal{L}_{SL_2})\otimes \mathrm{pr}_2^*(\mathcal{L}_{SL_2})^{3}.
\end{array}
$$
We get
$$j^*(\mathcal{L}_{G_2})=\mathcal{L}_{SL_2}\boxtimes\mathcal{L}_{SL_2}^{3}.$$
\item
The morphism $\mathrm{j} : \mathcal{M}_{SL_2} \times \mathcal{M}_{SL_2}\rightarrow \mathcal{M}_{G_2}$ is invariant under the action of $JC[2]$:
$(E,F)\in \mathcal{M}_{SL_2} \times \mathcal{M}_{SL_2}$ and $\alpha \in JC[2]$ then the rank-$7$ vector bundle associated by $\mathrm{j}$ to $(E\otimes \alpha,F\otimes \alpha)$ is
$$
\begin{array}{rl}
&(E\otimes \alpha)^* \otimes (F\otimes \alpha)  \oplus \mathrm{End}_0(F\otimes \alpha)\\
=&E^*\otimes F\otimes \alpha^*\otimes \alpha \oplus \mathrm{End}_0(F\otimes \alpha),\\
=&E\otimes F\oplus \mathrm{End}_0(F)= \mathrm{j}(E,F).
\end{array}
$$

Therefore, the image of $\Psi$ is contained in the expected vector space
$[H^0(\mathcal{M}_C(SL_2), \mathcal{L}_{SL_2})\otimes  H^0(\mathcal{M}_C(SL_2), \mathcal{L}_{SL_2}^{3})]_0$. 
\end{enumerate}
\end{proof}

Before going further in the study of the morphism $\mathrm{j}$, we compare the dimensions of the involved sets.
\begin{lem}
\label{Lem dim[H0(MSL2,LSL2) and H0(MSL2, LSL2 3)]0 = dim H0(MG2,LG2)}
The dimension of the space $\left[H^0(\mathcal{M}_C(SL_2), \mathcal{L}_{SL_2})\otimes  H^0(\mathcal{M}_C(SL_2), \mathcal{L}_{SL_2}^{3})\right]_0$ equals the dimension of $H^0(\mathcal{M}_C(G_2),\mathcal{L}_{G_2})$.
\end{lem}
\begin{proof}
By Proposition \ref{Prop h0(G2,L) h0(SL2,L3)},\eqref{h0(SL2,L3)=2g h0(G2,L)}, we notice the remarkable following relation:
\begin{align}
\label{Equation h Msl2, L3 = 2^g h Mg2, L}
2^g h^0\left(\mathcal{M}_C(G_2),\mathcal{L}\right)=h^0(\mathcal{M}_C(SL_2),\mathcal{L}^{3}).
\end{align}
So, as $h^0(\mathcal{M}_C(SL_2),\mathcal{L})=2^g$ (see \cite{BeauvilleI}), we get
 $$
\begin{array}{ll}
\dim \left(\left[H^0(\mathcal{M}_C(SL_2), \mathcal{L}_{SL_2})\otimes  H^0(\mathcal{M}_C(SL_2), \mathcal{L}_{SL_2}^{3})\right]_0 \right)&\\
=\frac{1}{2^{2g}}\times h^0(\mathcal{M}_C(SL_2), \mathcal{L}_{SL_2})\times  h^0(\mathcal{M}_C(SL_2), \mathcal{L}_{SL_2}^{3}) &  \text{because}\ |\,JC[2]\,|=2^{2g},\\
=\frac{1}{2^{2g}}\times 2^{g}\times 2^{g}h^0(\mathcal{M}_C(G_2), \mathcal{L}_{G_2}) &\text{by}\ \eqref{Equation h Msl2, L3 = 2^g h Mg2, L},\\
=h^0(\mathcal{M}_C(G_2), \mathcal{L}_{G_2}).&
\end{array}$$
\end{proof}

All the following results are based on the \textit{cubic normality conjecture}. Its statement is:

\begin{conj}
\label{Conj surjection  HO(MSL2,L) in H0(MSL2, L3)}
For a general curve $C$, the multiplication map 
$$\eta : \mathrm{Sym}^3 H^0(\mathcal{M}_C(SL_2), \mathcal{L}_{SL_2}) \rightarrow H^0(\mathcal{M}_C(SL_2), \mathcal{L}_{SL_2}^{3})$$
is surjective.
\end{conj}
When the previous map $\eta$ is surjective, we say that the curve $C$ satisfies \textit{cubic normality}.
\begin{prop}
\label{Prop normalité cubique vraie}
Cubic normality holds for
all curve of genus $2$, all non hyper-elliptic curve of genus $3$ and all curve of genus $4$ without effective theta-constant.
\end{prop}

\begin{proof}
For a curve of genus $2$, $\mathcal{M}_C(SL_2)$ is isomorphic to $\PP^3$
and $\mathcal{L}_{SL_2}$ to $\mathcal{O}(1)$ (see \cite{NarasimhanRamanan}).
A non-hyper-elliptic curve of genus $3$ is a Coble quartic (see \cite{NR}).
The cubic normality is true in both of these cases.
For a general curve of genus $4$ without effective theta-constant, 
cubic normality is proved in Theorem
 $4.1$ of \cite{OxburyPauly}.
\end{proof}

When this conjecture is true, we get this theorem:

\begin{thm}
\label{Thm : Conj implique isom}
Let $C$ be a curve of genus at least $2$ without effective theta-constant and satisfying the cubic normality and let $\Psi$ be the map defined in Theorem \ref{Theorem definition Psi}:
$$\Psi : H^0(\mathcal{M}_C(G_2),\mathcal{L}_{G_2})\rightarrow 
\left[H^0(\mathcal{M}_C(SL_2), \mathcal{L}_{SL_2})\otimes  H^0(\mathcal{M}_C(SL_2), \mathcal{L}_{SL_2}^{3})\right]_0. $$
\begin{enumerate}
\item
 \label{item Proposition Conj normalité cubique implique Psi isomorphisme}
 The map $\Psi$ is an isomorphism,
\item
\label{fait 2 H0(MG2, LG2) engendré par les rho1 (delta kappa)}
The space of generalized $G_2$-theta functions $H^0(\mathcal{M}_C(G_2), \mathcal{L})$ is linearly generated by the divisors
$\rho_0^*(\Delta_\kappa)$ for $\kappa$ even theta-characteristic,
where $\rho_0$ is the extension morphism 
$$\rho_0 : \mathcal{M}_C(G_2) \rightarrow \mathcal{M}_C(SO_7).$$
\end{enumerate}
\end{thm}

\begin{proof}
\begin{enumerate}
\item
According to the dimension equality proved in Lemma \ref{Lem dim[H0(MSL2,LSL2) and H0(MSL2, LSL2 3)]0 = dim H0(MG2,LG2)}, 
it suffices to prove the surjectivity of $\Psi$.

Denote by $V$ the vector space $H^0(\mathcal{M}_C(SL_2), \mathcal{L})$.

Using the notations of \cite{BeauvilleII}, we associate to each even theta-characteristic $\kappa$ an element
$d_\kappa$ of $H^0(\mathcal{M}_C(SL_2), \mathcal{L}^{\otimes 2})$ and an element $\xi_\kappa $ of $V\otimes V$.
 For each even theta-characteristic $\kappa$, $d_\kappa$ is the section of $H^0(\mathcal{M}_C(SL_2), \mathcal{L}^{\otimes 2})$ such that $D_\kappa$ is the divisor of the zeros of $d_\kappa$, where $D_\kappa=\{S \in M_C(SL_2)\ |\  h^0(C ,\mathrm{End}_0(S)\otimes \kappa)>0\}$.
Consider the following maps:
$$
\begin{array}{rcl}
\rho_0^* :&H^0( \mathcal{M}_C^+(SO_7), \mathcal{L})&\longrightarrow H^0(\mathcal{M}_C(G_2),\mathcal{L})\\
\text{and}\quad \beta :&[ V \otimes V \otimes H^0(\mathcal{M}_C(SL_2), \mathcal{L}^{2})]_0&\longrightarrow [V \otimes H^0(\mathcal{M}_C(SL_2), \mathcal{L}^{3}]_0
\end{array}
$$
where $\beta(A,B,D)=(A,BD)$. 
For any even theta-characteristic $\kappa$, the image $\Psi(\rho_0^*(\Delta_\kappa))$ equals $\beta(\xi_\kappa\otimes d_\kappa)$. Indeed, $\Psi$ is induced by:
$$
\begin{array}{rrcl}
\mathrm{j}:&\mathcal{M}_C(SL_2)\times\mathcal{M}_C(SL_2) &\rightarrow &\mathcal{M}_C(G_2) \\
&(E,F) & \mapsto &\mathrm{Hom}(E,F)  \oplus \mathrm{End}_0(F);
\end{array}
$$
the pull-back $\Psi(\rho^*_0(\Delta_\kappa))$ is the sum of two divisors:
$$
\begin{array}{rcl}
\Delta_1 & =&\{(E,F)\in \mathcal{M}_C(SL_2)\times\mathcal{M}_C(SL_2) \ |\ h^0(C,\mathrm{End}_0(F) \otimes \kappa) >0 \},\\
\Delta_2 & =&\{(E,F)\in \mathcal{M}_C(SL_2)\times\mathcal{M}_C(SL_2) \ |\ h^0(C, \mathrm{Hom}(E,F) \otimes \kappa >0 \}.
\end{array}
$$
In addition, $\mathcal{O}(\Delta_1)=\mathcal{O}_C \boxtimes \mathcal{L}^{2}$ and $\mathcal{O}(\Delta_2)= \mathcal{L} \boxtimes \mathcal{L}$ (see \cite{BeauvilleII}) and more precisely:
 $$\Delta_1 = \text{Zeros}(d_\kappa) \text{ et } \Delta_2= \text{Zeros}(\xi_\kappa) .$$
 When the curve $C$ is of genus at least $2$ without effective theta-constant,
it is proved in \cite{BNR} that the map
$$
\begin{array}{rrcl}
\varphi_0^* : & \mathrm{Sym}^2 V &\longrightarrow& H^0(\mathcal{M}_C(SL_2), \mathcal{L}^{2})\\
& \xi_\kappa &\mapsto& d_\kappa 
\end{array}
$$
is an isomorphism. We identify $\mathrm{Sym}^2 V$ with the invariant space of $V\otimes V$ under the involution $a\otimes b\mapsto b\otimes a$. 
 By Theorem $1.2$ and Proposition A.$5$ of \cite{BeauvilleII}, the set $\{d_\kappa\ |\ \kappa \in\Theta^{\mathrm{even}}(C) \}$ is a basis of $H^0(\mathcal{M}_C(SL_2), \mathcal{L}^{2})$ and $\{\xi_\kappa\ |\ \kappa \in\Theta^{\mathrm{even}}(C)\} $ is a basis of $\mathrm{Sym}^2 V$.
 Then, the vector space $[ V \otimes V \otimes H^0(\mathcal{M}_C(SL_2), \mathcal{L}^{2})]_0$ is generated by $\{\xi_\kappa \otimes d_\kappa \ |\ \kappa \in\Theta^{\mathrm{even}}(C) \}$. 
Thus, to prove the surjectivity of the map $\Psi$, it is sufficient to show the surjectivity of the map $\beta$.
Consider the following diagram:
\begin{diagram}
V\otimes H^0(\mathcal{M}_C(SL_2),\mathcal{L}^{2})&\rTo &H^0(\mathcal{M}_C(SL_2),\mathcal{L}^{3})\\
\uTo &&\uTo_{\eta} \\
V \otimes V\otimes V &\rTo & \mathrm{Sym}^3V.
\end{diagram}
With the hypothesis of cubic normality, the map $\eta$ is surjective. Therefore the map $V\otimes H^0(\mathcal{M}_C(SL_2),\mathcal{L}^{2})  \rightarrow H^0(\mathcal{M}_C(SL_2),\mathcal{L}^{3})$ is also surjective. By restriction to invariant sections under the action of $JC[2]$, $\beta$ is surjective.

The map $\Psi$ is thus an isomorphism.
\item
The point \eqref{fait 2 H0(MG2, LG2) engendré par les rho1 (delta kappa)} is a consequence of the previous facts:
for each element $\kappa \in  \Theta^{\mathrm{even}}(C)$, the image of $\rho_0^*(\Delta_\kappa)$ by $\Psi$ is $\xi_\kappa \otimes d_\kappa$. 
As $\{\xi_\kappa \otimes d_\kappa \ |\ \kappa \in  \Theta^{\mathrm{even}}(C)\}$ generates $[ V \otimes V \otimes H^0(\mathcal{M}_C(SL_2), \mathcal{L}^{2})]$, the set $\{\rho_0^*(\Delta_\kappa)\ |\ \kappa \in  \Theta^{\mathrm{even}}(C)\}$ generates $H^0(\mathcal{M}_C(G_2), \mathcal{L})$.
\end{enumerate}
\end{proof}

\begin{rem}
By Proposition \ref{Prop normalité cubique vraie}
the linear map $\Psi$ is an isomorphism for
each curve of genus $2$, each non hyperelliptic curve of genus $3$ and each curve of genus $4$ without effective theta-constant.
\end{rem}

\appendix
\label{appendix}

\section{The octonion algebra}
\label{appendix The octonion algebra}

Let $\mathcal{B}_0=\{e_1, \dots, e_7\}$ be the canonical basis of the
 subalgebra of the imaginary part of octonions.The multiplication rules are:
 \begin{itemize}
\item
$\forall i \in \{1, \dots , 7\},\quad e_i^2=-1$,
\item
$e_ie_j=-e_je_i=e_k$ when $(e_i, e_j, e_k)$ are three points on the same edge of on the oriented Fano diagram.
\end{itemize}

\begin{figure}[!h]

\begin{picture}(115,90)(0,-90)

\gasset{AHangle=23.5,AHLength=3.76,AHlength=2.8}

\drawqbezier[AHnb=2,AHdist=15.0,AHLength=3.76](39,-44,57.45,-24.1,75.9,-44)
\drawqbezier[AHnb=2,AHdist=15.0,AHLength=3.76](75.9,-44,85.3,-67,57.5,-76)
\drawqbezier[AHnb=2,AHdist=15.0,AHLength=3.76](57.5,-76,29.6,-67,39,-44)


\drawline[AHnb=4,AHdist=18.0](94.3,-76)(20.6,-76)
\drawline[AHnb=4,AHdist=18.0](20.6,-76)(57.5,-12.1)
\drawline[AHnb=4,AHdist=18.0](57.5,-12.1)(94.3,-76)
\drawline[AHnb=2,AHdist=5.0](39,-44)(57.5,-54.7)
\drawline[AHnb=2,AHdist=5.0](75.9,-44)(57.5,-54.7)
\drawline[AHnb=2,AHdist=5.0](57.5,-76)(57.5,-54.7)
\drawline[AHnb=2,AHdist=10.0](57.5,-54.7)(57.5,-12.1)
\drawline[AHnb=2,AHdist=10.0](57.5,-54.7)(20.6,-76)
\drawline[AHnb=2,AHdist=10.0](57.5,-54.7)(94.3,-76)
\node[Nfill=y,fillgray=1.0,NLangle=0.0,Nadjustdist=2.0,Nw=10.0,Nh=10.0,Nmr=5.0](n1)(39.0,-44.0){$e_1$}

\node[Nfill=y,fillgray=1.0,NLangle=0.0,Nadjustdist=2.0,Nw=10.0,Nh=10.0,Nmr=5.0](n2)(75.9,-44.0){$e_2$}

\node[Nfill=y,fillgray=1.0,NLangle=0.0,Nadjustdist=2.0,Nw=10.0,Nh=10.0,Nmr=5.0](n3)(57.5,-76.0){$e_3$}

\node[Nfill=y,fillgray=1.0,NLangle=0.0,Nadjustdist=2.0,Nw=10.0,Nh=10.0,Nmr=5.0](n4)(57.5,-54.7){$e_4$}

\node[Nfill=y,fillgray=1.0,NLangle=0.0,Nadjustdist=2.0,Nw=10.0,Nh=10.0,Nmr=5.0](n5)(57.5,-12.1){$e_5$}

\node[Nfill=y,fillgray=1.0,NLangle=0.0,Nadjustdist=2.0,Nw=10.0,Nh=10.0,Nmr=5.0](n6)(20.6,-76.0){$e_6$}

\node[Nfill=y,fillgray=1.0,NLangle=0.0,Nadjustdist=2.0,Nw=10.0,Nh=10.0,Nmr=5.0](n7)(94.3,-76.0){$e_7$}

\end{picture}

\caption{Fano Diagram}
\end{figure}
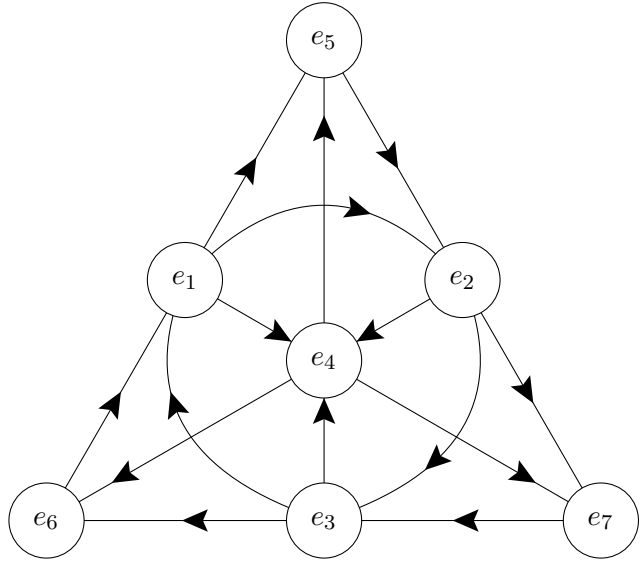


We introduced the basis $\mathcal{B}_1=\{y_1, \dots ,y_7\}$ obtained by 
basis change by
the change of basis matrix 
$P=\frac{\sqrt{2}}{2} \left( 
\begin{array}{ccc|ccc|c}
0&1&0&0&1&0&0\\
1&0&0&1&0&0&0\\
0&0&1&0&0&1&0\\
\hline
0&-i&0&0&i&0&0\\
-i&0&0&i&0&0&0\\
0&0&-i&0&0&i&0\\
\hline
0&0&0&0&0&0&\sqrt{2}\\
\end{array}
\right).
$
The canonical quadratic form on $\mathrm{Im} \OO$ expressed in the basis $\mathcal{B}_1$ is
$
Q=~
\left(
\begin{array}{r|c|l}
			 0 		& I_3	&0 \\
\hline I_3 	&0 		& 0 \\
\hline 0		&0 		&1
\end{array}
\right).
$
In the basis $\mathcal{B}_2=\{y_2,y_3,y_4,y_5,y_6,y_1,y_7\}$, the multiplication table is
\begin{align}
\label{tabular B2}
\begin{array}{|c||c|c|c||c|c|c||c|}
\hline
\nearrow & y_2 & y_3 & y_4 & y_5 & y_6 & y_1 & y_7\\
\hline
\hline
 y_2 & 0& - \sqrt{2}  y_1 & \sqrt{2}  y_6 & -1+i y_7 &0&0&-i y_2\\
 \hline
 y_3& \sqrt{2} y_1&0&-\sqrt{2} y_5&0&-1+i y_7&0&-i y_3 \\
 \hline
 y_4&-\sqrt{2} y_6& \sqrt{2} y_5&0&0&0&-1+i y_7&-i y_4\\
 \hline
 \hline
 y_5&-1-i y_7&0&0&0&-\sqrt{2} y_4&\sqrt{2} y_3&i y_5\\
 \hline
  y_6&0&-1-i y_7&0&\sqrt{2} y_4&0&-\sqrt{2} y_2&i y_6\\
  \hline
 y_1&0&0&-1-i y_7&-\sqrt{2} y_3&\sqrt{2} y_2&0&i y_1\\
 \hline
 \hline
  y_7& i y_2&i y_3&i y_4&-i y_5&-i y_6&-i y_1&-1\\
  \hline 
 \end{array}
 \end{align}
In the basis $\mathcal{B}_3=\{y_1,y_2,y_4,y_5,y_3,y_6,y_7\}$, the multiplication table is
\begin{align}
\label{tabular B3}
\begin{array}{|c||c|c||c|c||c|c|c|}
\hline
\nearrow	& y_1 & y_2 & y_4 & y_5 & y_3 & y_6 & y_7 \\
\hline
\hline
y_1 & 0 & 0 & -1-i y_7 & - \sqrt{2} y_3 & 0 &\sqrt{2}y_2  &i y_1 \\
\hline
y_2 & 0 & 0 & \sqrt{2} y_6 & -1+ i y_7 & -\sqrt{2} y_1 & 0 & -i y_2 \\
\hline
\hline
y_4 & -1+i y_7 & -\sqrt{2} y_6 & 0 & 0 & \sqrt{2} y_5 & 0 & -i y_4 \\
\hline
y_5 & \sqrt{2} y_3 & -1-i y_7 & 0 & 0 & 0 &-\sqrt{2} y_4 & i y_5 \\
\hline
\hline
y_3 & 0 & \sqrt{2} y_1 & - \sqrt{2}y_5 & 0 & 0 & -1+i y_7 & -i y_3 \\
\hline
y_6 & - \sqrt{2} y_2 & 0 & 0 & \sqrt{2} y_4 & -1-i y_7 & 0 & i y_6 \\
\hline
y_7 & -i y_1 & i y_2 & i y_4 & -i y_5 & iy_3 & -iy_6 & -1 \\
\hline
\end{array}
\end{align}

\section{Computation of $h^0(\mathcal{M}_C(G_2),\mathcal{L})$}
\label{appendix h0(G2,L)}

First, we recall the Verlinde Formula.

\begin{prop}[Verlinde Formula]
Let $G$ be a Lie group with Lie algebra $\got{g}$ of classical type or type $\g$,
$\mathcal{L}$ be the ample canonical line bundle on $\mathcal{M}_C(G)$
and $i$ be a positive integer. The integer $h^0(\mathcal{M}_C(G),\mathcal{L}^i)$ is given by the following relation:
 $$h^0(\mathcal{M}_C(G),\mathcal{L}^i)=(\#T_i)^{g-1}\sum_{\mu\in \mathcal{P}_i}\prod_{\alpha\in \Delta_+}\left[2 \sin\left(\frac{\pi\langle \alpha,\mu+\rho \rangle}{i+g^*}\right)\right]^{2-2g}.$$
where
$$
\begin{array}{rclcrcl}
\#T_i&=&(i+g^*)^{\mathrm{rk} ( \mathfrak{g})} \#(\mathcal{P}/Q) \#(Q/Q_{\text{lg}}),&\ & \langle \cdot, \cdot \rangle &\text{is}&\text{the Killing form},\\
&&\text{ where } \mathrm{rk}(\mathfrak{g})\text{ is the rank of } \got{g}, && \mathcal{P}_i&=&\{\text{dominant weights } \mu \  | \ \langle \mu, \theta \rangle \leq i\},\\
\mathcal{P} &\text{is}& \text{the weight lattice},&&\theta & \text{is}&\text{the maximal positive root},\\
Q &\text{is}& \text{the root lattice},&&\Delta_+&=&\{\text{positive roots of }\got{g}\}\\
Q_{\text{lg}} &\text{is}& \text{the long root lattice},&& g^* &\text{is}& \text{the dual Coxeter number of the group }G .\\
\rho &=&\frac{1}{2}\sum_{\alpha_j \in \Delta_+} \alpha_j,&& &&
\end{array}
$$
\end{prop}
\begin{proof} See the survey \cite{Sorger}.
\end{proof}
The root system of $G_2$ has six positive roots $\alpha_1,\dots, \alpha_6$, with two simple roots, called $\alpha_1$ and $\alpha_2$.

For $H^0(\mathcal{M}_C(G_2),\mathcal{L})$, the data used in the Verlinde Formula are 
$g^*=4$, $ \mathrm{rk}(\got{g}_2)=2 $, $\#T_1=(1+4)^2\times 1 \times 3=75$, $\rho =5\alpha_1 + 3\alpha_2$, $\theta =\alpha_6=3\alpha_1+2\alpha_2$
and the fundamental weights are $\varpi_1=\alpha_5$ and $\varpi_2=\alpha_6$.

To describe $ \mathcal{P}_1$, we describe the Killing form.
The angle formed between $\alpha_1$ and $\alpha_2$ is $5 \pi/6$ and the ratio $\|\alpha_2\|/\|\alpha_1\|= \sqrt{3}$.
In order to normalize the Killing form, we impose the norm of maximal positive root $\|\theta\|$ equals $2$ . As $\theta =\alpha_6$, $\|\alpha_2\|=\|\theta\|=2$ and then $\|\alpha_1\|=2/3$ and $\langle \alpha_1,\alpha_2 \rangle=-1$.
The evaluations of the Killing form on $\varpi_i$ and $\theta$ are
$\langle \varpi_1, \theta \rangle =1$ and $\langle \varpi_2, \theta \rangle =2$; so
 $\mathcal{P}_1=\{0, \varpi_1\}$.

The evaluation of the Killing form on each positive root and $\rho$ added to each value of $\mathcal{P}_1$ are the following:
$$
\begin{array}{rclcrcl}
\langle \alpha_1 , \rho \rangle &=& 1/3 & \text{and} &   \langle \alpha_1 , \varpi_1+\rho \rangle &=& 2/3  ,    \\
\langle \alpha_2 , \rho \rangle &=&   1, &&	\langle \alpha_2 , \varpi_1+\rho \rangle &=&    1        ,       \\
\langle \alpha_3 , \rho \rangle &=& 4/3, &&	\langle \alpha_3 , \varpi_1+\rho \rangle &=&    5/3     , 		\\
\langle \alpha_4 , \rho \rangle &=& 5/3,	&&	\langle \alpha_4 , \varpi_1+\rho \rangle &=&    7/3     , 		\\
\langle \alpha_5 , \rho \rangle &=&   2,	&& 	\langle \alpha_5 , \varpi_1+\rho \rangle &=&    3     	,	\\
\langle \alpha_6 , \rho \rangle &=&   3,	&&  \langle \alpha_6 , \varpi_1+\rho \rangle &=&     4   	.	\\
\end{array}
$$

So, by the Verlinde formula, the dimension $h^0(\mathcal{M}_C(G_2),\mathcal{L})$ is 
$$\begin{array}{rl}
h^0(\mathcal{M}_C(G_2),\mathcal{L}) =&\left(\frac{2^{10}}{25}\right)^{1-g}
\Big[\left[\sin^2\left(\frac{\pi}{15}\right)\sin^2\left(\frac{4\pi}{15}\right)\sin^2\left(\frac{\pi}{5}\right)\sin^4\left(\frac{2\pi}{5}\right)\right]^{1-g}\\
&+\left[\sin^2\left(\frac{2\pi}{15}\right)\sin^2\left(\frac{7\pi}{5}\right)\sin^2\left(\frac{2\pi}{5}\right)\sin^4\left(\frac{\pi}{5}\right)\right]^{1-g}
\Big].
\end{array}
$$
To obtain a compact formula, we express these trigonometric products in $\mathbb{Q}(\sqrt{5})$:
$$
\begin{array}{clccl}
\sin^2\left(\frac{\pi}{5}\right)&=\frac{1}{8}(5-\sqrt{5}),&\quad&
\sin^2\left(\frac{\pi}{15}\right)\sin^2\left(\frac{4\pi}{15}\right)&= \frac{1}{2^5} (3-\sqrt{5}),\\
\sin^2\left(\frac{2\pi}{5}\right)&=\frac{1}{8}(5+\sqrt{5}),&&
\sin^2\left(\frac{2\pi}{15}\right)\sin^2\left(\frac{7\pi}{15}\right)&= \frac{1}{2^5} (3+\sqrt{5}),\\
\sin^4\left(\frac{\pi}{5}\right)&=\frac{5}{2^5}(3-\sqrt{5}),&&
\sin^4\left(\frac{2\pi}{5}\right)&=\frac{5}{2^5}(3+\sqrt{5}).
\end{array}
$$
So,
$$
\begin{array}{cl}
&h^0(\mathcal{M}_C(G_2),\mathcal{L}) \\
=& \left(\frac{2^{10}}{25}\right)^{1-g} 
\Big[
\left[ 
\left(\frac{1}{2^5} \right) 
\left(\frac{1}{2^3} \right)
\left(\frac{5}{2^5}\right) 
\left(3-\sqrt{5}\right)
\left(5-\sqrt{5}\right) 
\left(3+\sqrt{5}\right) 
\right]^{1-g} \\
&+
\left[ 
\left(\frac{1}{2^5} \right)
\left(\frac{1}{2^3} \right)
\left( \frac{5}{2^5}\right) 
\left(3+\sqrt{5}\right)
\left(5+\sqrt{5}\right)
\left(3-\sqrt{5}\right) 
\right]^{1-g}
\Big],\\
=& \left(\frac{5-\sqrt{5}}{10}\right)^{1-g} +\left(\frac{5+\sqrt{5}}{10}\right)^{1-g}
= \left(\frac{5+\sqrt{5}}{2}\right)^{g-1} +\left(\frac{5-\sqrt{5}}{2}\right)^{g-1} .
\end{array}
$$

\section{Computation of $h^0(\mathcal{M}_C(SL_2),\mathcal{L}^3)$}
\label{appendix h0(SL2,L3)}
Using the previous notations, the evaluation on $\alpha$ and each element of $\mu + \rho $, where $\mu \in\mathcal{P}_3$, are the following:
$$
\begin{array}{rclcrcl}
\langle \alpha, \rho \rangle          &=&1, & \text{and} &\langle \alpha, 2\varpi_1+\rho \rangle &=&3,\\
\langle \alpha, \varpi_1+\rho \rangle  &=&2,&&\langle \alpha, 3\varpi_1+\rho \rangle  &=&4.\\
\end{array}
$$
By the Verlinde Formula,
$$
\begin{array}{cl}
&h^0(\mathcal{M}_C(SL_2),\mathcal{L}^{\otimes 3})\\
=&
(10)^{g-1}(2)^{2-2g}\left[\left[\sin\left(\frac{\pi}{5}\right)\right]^{2-2g} + \left[\sin\left(\frac{2\pi}{5}\right)\right]^{2-2g} +\left[\sin\left(\frac{3\pi}{5}\right)\right]^{2-2g} +\left[\sin\left(\frac{4\pi}{5}\right)\right]^{2-2g}\right],\\
=&
\left(\frac{5}{2}\right)^{g-1} \left[2\left[\sin^2\left(\frac{\pi}{5}\right)\right]^{1-g} + 2\left[\sin^2\left(\frac{2\pi}{5}\right)\right]^{1-g}\right],\\
=&2\left(\frac{5}{2}\right)^{g-1}\left(\left(\frac{8}{5-\sqrt{5}}\right)^{g-1} + \left(\frac{8}{5+\sqrt{5}}\right)^{g-1}\right),\\
=&2^g\left[\left(\frac{5+\sqrt{5}}{2}\right)^{g-1}+
\left(\frac{5-\sqrt{5}}{2}\right)^{g-1}\right] .
\end{array}
$$

\section*{Acknowledgements}

I would like to thank Christian PAULY for all his valuable suggestions and comments.

\bibliography{Biblio}
\bibliographystyle{smfalpha}
\end{document}